\documentclass[11pt]{amsart}

\usepackage{amsmath}
\usepackage{amssymb}
\usepackage{amscd}
\usepackage{a4wide}
\usepackage{hyperref}
\usepackage[utf8]{inputenc}
\usepackage{color}

\newtheorem{thm}{Theorem}[section]
\newtheorem{cor}[thm]{Corollary}
\newtheorem{lem}[thm]{Lemma}
\newtheorem{prop}[thm]{Proposition}
\theoremstyle{definition}

\newtheorem{rem}[thm]{Remark}
\newtheorem{example}[thm]{Example}
\newtheorem{examples}[thm]{Examples}

\newcommand{\Z}{\mathbb{Z}}

\newcommand{\Q}{\mathbb{Q}}
\newcommand{\N}{\mathbb{N}}
\newcommand{\C}{\mathbb{C}}

\title[On cyclic essential extensions of simple modules]{On cyclic essential extensions of simple modules over differential operator rings}

\author{Alveri Sant'Ana and Robson Vinciguerra}

\address{Alveri Sant'Ana, Instituto de Matemática, Universidade Federal do Rio Grande do Sul, Brazil}
\email{alveri@mat.ufrgs.br}

\address{Robson Vinciguerra, Instituto de Matemática, Universidade Federal do Rio Grande do Sul, Brazil}
\email{robson.vinciguerra@ufrgs.br}


\thanks{The last author was partially supported by CNPq, Brazil. Some results were obtained during the visit of the last author at University of Porto, Portugal. He would like to thanks the hospitality of the mathematical department of University of Porto}

\begin{document}

\begin{abstract}
In this paper we discuss under which conditions cyclic essential extensions of simple modules over a differential operator ring $R[\theta; \delta]$ are Artinian. In particular, we study the case when $R$ is either $\delta$-simple or $\delta$-primitive. Furthermore, we obtain important results when $R$ is an affine algebra of Kull dimension 2. As an application we characterize the differential operator rings $\C[x, y][\theta; \delta]$ for which cyclic essential extensions of simple modules are Artinian.
\end{abstract}

\maketitle


\section*{Introduction}

A Noetherian ring $S$ whose cyclic essential extensions of simple $S$-modules are Artinian are known as rings which satisfy the {\it property ($\diamond$)}. Such rings began to be studied independently around 1974 when Jategaonkar used this property to answer the Jacobson's conjecture in the affirmative, for
fully bounded Noetherian rings. The property ($\diamond$) has been investigated in several classes of rings, for example: Noetherian commutative rings \cite{Matlis} and \cite{Matlis2}; Noetherian domains of Krull dimension one (see \cite[Theorem 10]{Krause}); the enveloping algebra ${\mathcal U}(sl_2(K))$ \cite{Dahlberg}; Noetherian down-up algebras \cite{CLP}; the quantum plane and the quantized Weyl algebras \cite{CM}.

The first example of a Noetherian domain not satisfying this property was published by Musson in 1980 \cite{Musson2}. The same author showed in \cite{Musson}, that differential operator rings $K[x][\theta, x^n\partial_x]$ not satisfy $(\diamond)$ whenever $n > 0$. In a recent paper by Carvalho, Hatipo\u{g}lu and Lomp \cite{CHL}, it was completely characterized the rings $K[x][\theta, \delta]$ which have the property ($\diamond$) (see \cite[Corollary 4.1]{CHL}). Furthermore, in the same paper it was also shown that if $\delta$ is a locally nilpotent derivation, then $R[\theta; \delta]$ satisfies $(\diamond)$ (see \cite[Proposition 2.1]{CHL}).

Our main purpose in this paper is to provide necessary and sufficient conditions for $R[\theta; \delta]$ to satisfy the property $(\diamond)$ in the following cases: $R$ is $\delta$-simple; $R[\theta; \delta]$ is a primitive ring and also whenever $R$ has Krull dimension 2.

We organize the paper as follows. In the first section, we fix some notations and we recall some basic concepts involving derivations and differential operator rings that will be useful throughout the text.

In the second section, we show that under certain conditions, if a Noetherian integral domain $R$ is $\delta$-simple, then $R[\theta; \delta]$ does not satisfy $(\diamond)$ (see Theorem \ref{aneldsimples}). As usual, the Krull dimension of $R$ will be denoted by $\mbox{K.dim}(R)$. When $R$ is $\delta$-simple and $\mbox{K.dim}(R) \leq 1$, $R[\theta; \delta]$ has the property $(\diamond)$ (see Proposition \ref{simplesdiamantedim1}). As a consequence, if a commutative $K$-algebra $R$ is either a unique factorization domain (UFD, for short) $\delta$-simple or an affine domain $\delta$-simple, we show that $R[\theta; \delta]$ satisfies $(\diamond)$ if and only if $\mbox{K.dim}(R) \leq 1$ (see Theorem \ref{dfudsimplesdiamante}).

In the third section, we show that if a commutative $K$-algebra $R$ is an affine domain, then $R$ is $\delta$-primitive (where $\delta \neq 0$) if and only if $R[\theta; \delta]$ is a primitive ring (see Theorem \ref{primitivoeqdprimitivo}). As a consequence, in the case when $R[\theta; \delta]$ is primitive (or equivalently, $R$ is $\delta$-primitive), we obtain that $R[\theta; \delta]$ satisfies  $(\diamond)$ if and only if $R$ is $\delta$-simple and $\mbox{K.dim}(R) \leq 1$ (see Corollary \ref{diamanteprimitivo}). Moreover, when $\mbox{K.dim}(R) = 1$, we characterize the differential operator rings $R[\theta; \delta]$ which satisfy the property $(\diamond)$ (see Theorem \ref{kdimumdiamante}) and we present an example of a differential operator ring $R[\theta; \delta]$ satisfying $(\diamond)$, where the derivation $\delta$ is not locally nilpotent (see Remark \ref{diamantenaolocal}).

In the last section, we consider $R$ a UFD of Krull dimension 2 which is an affine $K$-algebra. Whenever $R$ does not have maximal ideals which are $\delta$-ideals, we provide necessary and sufficient conditions under which $R[\theta; \delta]$ satisfies $(\diamond)$ (see Theorem \ref{primitivodiamante}). In this case when $R$ has maximal ideals which are $\delta$-ideals, we give sufficient conditions for that $R[\theta; \delta]$ does not satisfy $(\diamond)$ (see Proposition \ref{contexemprimitivodiamante}). Moreover, the differential operator rings $\C[x, y][\theta; \delta]$ that are primitives were characterized (see Proposition \ref{caractdprimitivo}). As an application we provide necessary and sufficient conditions for $\C[x, y][\theta; \delta]$ to satisfy $(\diamond)$ (see Theorem \ref{primitivodiamante2}).


\section{Preliminaries}
Throughout this paper $K$ will denote a field of characteristic zero and $R$ will be a commutative ring. An additive map $\delta: R \rightarrow R$ is called a {\it derivation} of $R$ if $$\delta(ab) = \delta(a)b + a\delta(b)$$ for all $a, b \in R$. If $R$ is a $K$-algebra and $\delta$ is a derivation of $R$ such that $\delta(\alpha a) = \alpha \delta(a)$ for all $\alpha \in K$ and $a \in R$, we say that $\delta$ is a {\it $K$-derivation}. If $R = K[x_{1}, \ldots, x_{n}]$, the polynomial ring in $n$ variables over $K$, any $K$-derivation of $K[x_{1}, \ldots, x_{n}]$ is of the form $$\delta = a_1 \partial_{x_{1}} + \cdots + a_n\partial_{x_{n}},$$ where $\partial_{x_{i}} = \partial/\partial_{x_{i}}$ is the partial derivative with respect to $x_i$ and $a_1, \ldots, a_n \in R$ (see \cite[Theorem 1.2.1]{Nowicki}).

An ideal $I$ of $R$ is said to be a {\it $\delta$-ideal} if $\delta(I) \subseteq I$. Of course, 0 and $R$ are $\delta$-ideals of $R$ and in the case when these are the only such ideals of $R$, we say that $R$ is {\it $\delta$-simple}. The set $$R^{\delta} = \{r \in R \,|\, \delta(r) = 0\}$$ is a subring of $R$ known as the {\it ring of constants} of $R$. Note that if a $K$-algebra $R$ is $\delta$-simple, then $R^{\delta}$ is a field and $R$ is a domain (see \cite[Propositions 13.1.1 and 13.1.2]{Nowicki}).

An element $a \in R$ is called a {\it Darboux element} with respect to $\delta$ if $\delta(a) = ba$ for some $b \in R$. If the context is clear we simply refer to $a$ as a Darboux element. Note that $a$ is a Darboux element if and only if $Ra$ is a non-zero $\delta$-ideal of $R$. Furthermore, whenever a $K$-algebra $R$ is a UFD and $a\in R$ is a Darboux element, it is easy to check that every irreducible factor of $a$ is also a Darboux element.

Let $I$ be a ideal of $R$ and set $$(I: \delta) = \{r \in R \,|\, \delta^n(r) \in I, \forall\, n \geq 0\}.$$ It is easy to see that $(I: \delta)$ is the biggest $\delta$-ideal of $R$ contained in $I$. Moreover, if $R$ is a $K$-algebra and $P$ is a prime ideal of $R$, $(P: \delta)$ is a prime $\delta$-ideal (see \cite[Proposition 1.1]{GW}).

We say that an ideal $P$ of $R$ is {\it $\delta$-prime} if $P$ is a $\delta$-ideal such that $P \neq R$ and for all $I$ and $J$ $\delta$-ideals of $R$ with $IJ \subseteq P$, then $I \subseteq P$ or $J \subseteq P$. The ring $R$ is said to be {\it $\delta$-prime} whenever $0$ is a $\delta$-prime ideal. Note that if $P$ is a prime ideal that is a $\delta$-ideal, then $P$ is $\delta$-prime. In the case when $R$ is a Noetherian $K$-algebra, the $\delta$-prime ideals of $R$ are precisely the prime ideals of $R$ that are $\delta$-ideals (see \cite[Corollary 1.4]{GW}).

Let $A$ be a multiplicative system of $R$, that is, $A$ is a multiplicatively closed set, $0 \not\in A$ and $1 \in A$. We will denote by $RA^{-1}$ the localization of $R$ at $A$ and by $\pi: R \rightarrow RA^{-1}$ the canonical homomorphism $r \mapsto r1^{-1}$. If $A = \{x^i \,|\, i \in \N \}$, for some $x \in R$ non-zero-divisor, we denote $RA^{-1}$ by $R_x$ and if $A = R \setminus P$, for some prime ideal $P$ of $R$, we write $R_P$ for $RA^{-1}$ and $P_P$ for the extended ideal $PA^{-1}$. Note that any derivation $\delta$ of $R$ extends uniquely to a derivation (still denoted by $\delta$) of $RA^{-1}$ via the quotient's rule $\delta \left(rx^{-1}\right) = (\delta(r)x - r\delta(x))x^{-2}$ for all $rx^{-1} \in RA^{-1}$. Furthermore, if $I$ is a $\delta$-ideal of $R$, the extended ideal $IA^{-1}$ is a $\delta$-ideal of $RA^{-1}$. Moreover, it is easy to check that $RA^{-1}$ is $\delta$-simple whenever $R$ is $\delta$-simple.

The {\it differential operator ring} of $R$ with respect to $\delta$, denoted by $S = R[\theta; \delta]$, is a free left (and right) $R$-module with basis $\{1, \theta, \theta^{2}, \ldots \}$. The elements of $S$ are polynomials in $\theta$ with coefficients in $R$ whose addition is the usual of polynomials and the multiplication extends from $R$ via the rule $\theta a = a \theta + \delta(a)$ for all $a \in R$. Let $f = \sum_{i=0}^n a_i \theta^{i}$ be a non-zero polynomial of $S$ with $a_n \neq 0$. The integer $n$ is called {\it degree of $f$} and will be denoted by $\mbox{deg}(f)$. As usual, we say that the degree of the zero of $S$ is $-\infty$. Moreover, the following identities are true: $$\theta^n a = \sum_{i=0}^n {n \choose i} \delta^{n-i}(a) \theta^{i} \ \ \ \ \mbox{and} \ \ \ \ a\theta^n = \sum_{i=0}^n (-1)^{i}{n \choose i} \theta^{n-i} \delta^i(a) \ \ \ \ \ \mbox{for all} \ a \in R.$$
Let $I$ be a $\delta$-ideal of $R$ and $\overline{\delta}$ the derivation on $R/I$ induced by $\delta$, that is, $\overline{\delta}(a + I) = \delta(a) + I$ for all $a \in R$. In this case $$S/SI \simeq R/I[\theta; \overline{\delta}].$$
If $R$ is a Noetherian $K$-algebra and $\mathcal{P}$ is any prime ideal of $S$, then $P = \mathcal{P} \cap R$ is a prime $\delta$-ideal of $R$ satisfying either $\mathcal{P} = SP$ or $\delta(R) \subseteq P$ (see \cite[Lemma 3.22]{GW_Livro}). Furthermore, whenever $\delta$ is non-zero and $\mathcal{P} \neq 0$, we must have also $P = \mathcal{P} \cap R \neq 0$ (see \cite[Lemma 3.18]{GW_Livro}).


\section{Commutative Noetherian $\delta$-simple Rings}

We say that $R$ is {\it $\delta$-primitive} if $R$ contains a maximal ideal which does not contain any non-zero $\delta$-ideals. In \cite{CHL}, Carvalho, Hatipo\u{g}lu and Lomp showed that given a differential operator ring $R[\theta; \delta]$ over a Noetherian integral domain without $\Z$-torsion $R$, if $R$ is $\delta$-primitive and $R[\theta; \delta]$ satisfy $(\diamond)$, then $R$ is $\delta$-simple. This fact motivated us to study property $(\diamond)$ whenever $R$ is $\delta$-simple. In this section, we shall give necessary and sufficient conditions for $R[\theta; \delta]$ to satisfy $(\diamond)$ whenever $R$ is $\delta$-simple. We begin by showing some lemmas that will be useful to construct a non-Artinian cyclic essential extension of a simple $R[\theta; \delta]$-module.

\begin{lem}\label{lema1}
	Let $R$ be an integral domain with derivation $\delta$ and $S = R[\theta; \delta]$. If $x \in R \setminus \{0\}$, then $R \cap S\theta x = 0$.
\end{lem}
\begin{proof}
	Suppose that there exists $a = (\sum_{i=0}^{n} b_i\theta^{i})\theta x \in R \cap S\theta x$ with $a, b_i \in R$ and $b_n \neq 0$. Then $$a = \sum_{i=0}^{n} b_i\theta^{i+1}x = b_n\theta^{n+1}x + \sum_{i=0}^{n-1} b_i\theta^{i+1}x = b_nx\theta^{n+1} + g$$	
	for some $g \in S$ such that $\mbox{deg}(g) < n+1$. By comparing the coefficients of all monomials in $\theta$, we obtain $b_nx = 0$. Since $R$ is a domain, we must have $x = 0$, which is a contradiction. Thus $\sum_{i=0}^{n} b_i\theta^{i} = 0$, and therefore $a = 0$.
\end{proof}

\begin{lem}\label{essencial}
	Let $R$ be a $K$-algebra which is an integral domain with a derivation $\delta$ and $S = R[\theta; \delta]$. Suppose that there exists $x \in R$ such that $\delta(x)$ is an invertible element in $R$ or $x$ is a prime element but not Darboux. Then $S/S\theta x$ is an essential extension of the left $S$-submodule $Sx/S\theta x$.
\end{lem}
\begin{proof}
	We claim that if $f = \sum_{i=0}^n a_i \theta^{i} \in S \setminus \{0\}$ and $x \nmid a_n$, then there exists $r \in R \setminus \{0\}$ such that $$x^{n+1}f - rx \in S\theta x.$$
	
	We will argue by induction on $n = \mbox{deg}(f)$. If $n = 0$, then $f = a_0 \in R$ and we can consider $r = a_0 \neq 0$. Then $xf-rx = xa_0 - a_0x = 0 \in S\theta x$.
		
		Suppose that $n > 0$ and assume the result is valid for every $g \in S$ such that  $\mbox{deg}(g) < n$. Then we have
		$$x^{n+1} f =  x^{n+1}\left(\sum_{i=0}^n a_i \theta^{i}\right)
		=  x^{n+1}(a_n\theta^n + a_{n-1}\theta^{n-1} + g_1)$$ for some $g_1 \in S$ such that $\mbox{deg}(g_1) < n-1$. Thus
		\begin{eqnarray*}
			x^{n+1} f & = & x^{n}[a_n(x\theta^n) + xa_{n-1}\theta^{n-1} + xg_1]\\
			& = &  x^{n}[a_n(\theta^nx - n\delta(x)\theta^{n-1} + g_2) + xa_{n-1}\theta^{n-1} + xg_1]
		\end{eqnarray*}
		for some $g_2 \in S$ such that $\mbox{deg}(g_2) < n-1$. Consequently,
		$$x^{n+1} f =  (x^{n}a_n\theta^{n-1})\theta x + x^{n}[(xa_{n-1}- na_n\delta(x))\theta^{n-1} + \underbrace{a_ng_2 +  xg_1}_{\mbox{deg} \,< \,n-1}].$$
		
		If $\delta(x)$ is an invertible element in $R$, then $x \nmid n\delta(x)a_n$ because $n\delta(x)$ is invertible in $R$ and $x \nmid a_n$. Hence, $x\nmid (xa_{n-1}- na_n\delta(x))$. On the other hand, if $x$ is a prime element such that $x \nmid \delta(x)$, then $x \nmid n\delta(x)a_n$ and so $x\nmid (xa_{n-1}- na_n\delta(x))$. Using the induction hypothesis, we can write $$x^{n}[(xa_{n-1}- na_n\delta(x))\theta^{n-1} + a_ng_2 +  xg_1] = h\theta x + rx$$ for some $h \in S$ and $r \in R \setminus \{0\}$. Thus
		$$x^{n+1} f = \underbrace{(x^{n}a_n\theta^{n-1})\theta x + h\theta x}_{\in \, S\theta x} + rx$$	
		and $x^{n+1} f - rx \in S\theta x$ with $r \in R \setminus \{0\}$, as claimed.
	
	Let $U$ be any non-zero left $S$-submodule of $S/S\theta x$. We will show that $U \cap (Sx/S\theta x) \neq 0$. In fact, let $f + S\theta x$ be a non-zero element of minimal degree of $U$, say $f = \sum_{i=0}^n a_i \theta^{i}$ with $a_n \neq 0$.
	
	If $n = 0$, then $f = a_0 \in R \setminus \{0\}$. Thus $xf = xa_0 \in Sx\setminus S\theta x$ since $xa_0\in R\setminus \{0\}$ and, by Lemma \ref{lema1}, $R\cap S\theta x = 0$. Hence, $0 \neq xf + S\theta x \in U \cap (Sx/S\theta x)$ and so $U \cap (Sx/S\theta x) \neq 0$.
	
	If $n > 0$, then $x \nmid a_n$. Otherwise, $a_n = bx$ for some $b \in R$ and $f =  b(x\theta^n) + g_1$ where $g_1 \in S$ and $\mbox{deg}(g_1) < n$. Then $$f = b(\theta^nx + g_2 )  +  g_1$$
	where $g_2 \in S$ with $\mbox{deg}(g_2) < n$. Hence, $$f = (b\theta^{n-1})\theta x + bg_2  +  g_1.$$
	This implies that $0 \neq f + S\theta x =  (bg_2  +  g_1) + S\theta x \in U$ with $\mbox{deg}\,(bg_2  +  g_1 ) < n$, contradicting our minimality assumption. Therefore $x \nmid a_n$.
	
	Now, by the above statement, there exists $r \in R \setminus \{0\}$ such that $x^{n+1} f - rx \in S\theta x$. Hence, $$x^{n+1}f + S\theta x = rx + S\theta x \in U \cap (Sx/S\theta x).$$ Note that, by Lemma \ref{lema1}, $rx + S\theta x \neq 0$. Therefore $U \cap (Sx/S\theta x) \neq 0$ as desired.
\end{proof}

Let $R$ be a ring with a derivation $\delta$ and $S = R[\theta; \delta]$. Then $R$ becomes a left $S$-module under the action $$\left(\sum_{i=0}^{n}a_i\theta^i\right)\cdot b = \sum_{i=0}^{n}a_i\delta^i(b)$$ for all $b \in R$ and $\sum_{i=0}^{n}a_i\theta^i \in S$. The map $\phi: S \rightarrow R$ defined by $$\phi\left(\sum_{i=0}^{n}a_i\theta^i\right) = \left(\sum_{i=0}^{n}a_i\theta^i\right) \cdot 1 = a_0$$ is an epimorphism of left $S$-modules with $\mbox{ker}(\phi) = S\theta$. Hence, $S/S\theta \,\simeq \, R$ as left $S$-modules and the $S$-submodules of $S/S\theta$ are precisely the left $\delta$-ideals of $R$.

As in \cite{GW_Krull}, we define the {\it $\delta$-Krull dimension} of $R$, denoted by $\mbox{$\delta$-K.dim}(R)$, as being the Krull dimension of left module $_S R$, that is, $\mbox{$\delta$-K.dim}(R) = \mbox{K.dim}_S(R)$.

\begin{lem}\label{nonartinian}
Let $R$ be a Noetherian $K$-algebra which is an integral domain with a derivation $\delta$ and $S = R[\theta; \delta]$. If $R$ is $\delta$-simple and $P$ is a non-maximal prime ideal of $R$, then $S/SP$ is a non-Antinian left $S$-module.
\end{lem}
\begin{proof}
	First assume that $\mbox{K.dim}(R/P)$ is finite. Since $P$ a non-maximal prime ideal of $R$, there exists a maximal ideal $M$ such that $M \supsetneq P$. Note that $SM$ is a proper left ideal of $S$, otherwise we would have $1 \in SM \cap R = M$, which is a contradiction. Hence, $S/SM \neq 0$, and therefore $\mbox{K.dim}_S(S/SM) \geq 0$. This implies that $$m = \mbox{max}\{\mbox{K.dim}_S(S/SQ) \ | \ Q \in \mbox{Spec}(R) \ \ \mbox{and}\ \ Q \supsetneq P\} \geq 0.$$ By \cite[Proposition 2.7]{GW_Krull}, we obtain $\mbox{K.dim}_S(S/SP) = m + 1 \geq 1$, so that $S/SP$ is non-Antinian.
	
	Suppose now that $\mbox{K.dim}(R/P)$ is infinite. Since $R$ is $\delta$-simple, $_SR$ is simple because the $S$-submodules of $_SR$ are precisely the $\delta$-ideals of $R$. Hence, $\mbox{$\delta$-K.dim}(R) = \mbox{K.dim}_S(R) = 0$. By \cite[Proposition 4.2]{GW_Krull}, we obtain $\mbox{K.dim}_S(S/SP) = \mbox{K.dim}(R/P)$ is infinite. Again, it follows that $S/SP$ is non-Antinian.
\end{proof}

Now we are able to give sufficient conditions to obtain non-Artinian cyclic essential extension of a simple $R[\theta; \delta]$-module, in the case when $R$ is $\delta$-simple.

\begin{thm}\label{aneldsimples}
Let $R$ be a Noetherian $K$-algebra which is an integral domain with a derivation $\delta$ such that $R$ is $\delta$-simple. If there exists $x \in R$ satisfying the following two conditions:
	\begin{enumerate}
		\item there exists a non-maximal prime ideal $P \subseteq R$ such that $x \in P$,
		\item $\delta(x)$ is an invertible element in $R$ or $x$ is a prime element of R,
	\end{enumerate}
	then $S = R[\theta; \delta]$ does not satisfy $(\diamond)$.
\end{thm}
\begin{proof}
	First note that $x \nmid \delta(x)$, otherwise $Rx$ would be a non-trivial $\delta$-ideal of $R$, a contradiction. Since either $\delta(x)$ is an invertible element in $R$ or $x$ is a prime element in $R$, by Lemma \ref{essencial}, it follows that $S/S\theta x$ is an essential extension of the left $S$-submodule $Sx/S\theta x$. Moreover, using the fact that $R$ is $\delta$-simple, it follows that $S/S\theta \simeq Sx/S\theta x$ is a simple left $S$-module.
	
	Now, since $P$ is a non-maximal prime ideal of $R$ and $R$ is $\delta$-simple, it follows that $S/SP$ is a non-Antinian left $S$-module, by Lemma \ref{nonartinian}. As $x \in P$, we have $S\theta x \subseteq SP \subseteq S$, and consequently $S/S\theta x$ is a non-Artinian essential extension of the simple module $Sx/S\theta x$.
\end{proof}

\begin{prop}\label{simplesdiamantedim1}
	Let $R$ be a Noetherian integral domain without $\Z$-torsion and $\delta$ be a derivation of $R$ such that $R$ is $\delta$-simple. If $\mbox{K.dim}(R) \leq 1$, then $S = R[\theta; \delta]$ satisfies $(\diamond)$.
\end{prop}
\begin{proof}	
	Suppose first that $R$ is a field. In this case, since $0$ is a maximal ideal of $R$ which is a $\delta$-ideal, from \cite[Theorem 2.10]{GW_Krull} it follows that $\mbox{K.dim}(S) = 1$. Now, assume that $\mbox{K.dim}(R) = 1$ and consider $M$ any maximal ideal of $R$. Since $R$ is $\delta$-simple, $\delta(M)\nsubseteq M$. Moreover, as $R$ is an integral domain without $\Z$-torsion, if $\mbox{char}(R/M) = p > 0$, then $Rp \subseteq M$ is a non-zero $\delta$-ideal, a contradiction. Hence, $\mbox{char}(R/M)=0$ and so  $\mbox{K.dim}(S) = 1$, by \cite[Theorem 2.10]{GW_Krull}. Therefore in both cases we have that $S$ is a Noetherian domain of Krull dimension 1, and consequently $S$ satisfies $(\diamond)$ (see \cite[Theorem 10]{Krause}).
\end{proof}

We will apply the general results obtained so far to UFDs and also to affine domains. For such rings, we get a full description of when $R[\theta; \delta]$ satisfies $(\diamond)$ provided $R$ is also $\delta$-simple. We start by recalling some basic facts of commutative algebra.

Given a commutative Noetherian ring $R$ and $I$ an ideal of $R$, the {\it grade (or depth)} of $I$, denoted by $G(I)$, is the length of a maximal regular sequence contained in $I$ (see \cite{K}). In general, we have $G(I) \leq \mbox{ht}(I)$, where $\mbox{ht}(I)$ is the {\it height of $I$} (also called {\it rank of $I$} or {\it codimension of $I$})  (see \cite[Theorem 132]{K}). A ring $R$ is {\it Cohen-Macaulay} if $G(M) = \mbox{ht}(M)$ for any maximal ideal $M$ of $R$. A local Noetherian ring $R$ with maximal ideal $M$ is said to be {\it regular} if the minimal number of generators of $M$ is equal to the Krull dimension of $R$. In such a case, any minimal set of generators for $M$ form a regular sequence, and therefore $G(M) = \mbox{ht}(M)$ (see \cite[Corollary 10.15]{E}). This shows that every regular local ring is a Cohen-Macaulay ring. Finally, a ring $R$ is said to be {\it regular} if $R_P$  is a regular local ring for any prime ideal $P$ of $R$.

\begin{thm}\label{dfudsimplesdiamante}
Let $R$ be a commutative $K$-algebra which is a Noetherian UFD or an affine domain with a derivation $\delta$ such that $R$ is $\delta$-simple. Then $$R[\theta; \delta] \ \mbox{satisfies} \ (\diamond) \ \ \mbox{if and only if} \ \  \mbox{K.dim}(R) \leq 1 .$$
\end{thm}
\begin{proof}
	($\Rightarrow$) Suppose first that $R$ is a UFD and $\mbox{K.dim}(R) > 1$, then every non-zero prime ideal contains a prime element. Hence, $R[\theta; \delta]$ does not satisfy $(\diamond)$, by Theorem \ref{aneldsimples}. 
	
	Assume now that $R$ is an affine domain such that $\mbox{K.dim}(R) = n > 1$. Thus there exists a maximal ideal $M$ of $R$ such that $\mbox{ht}(M) = n$. Since $R$ is a $\delta$-simple affine algebra, it follows from \cite[Theorems 3 and 5]{Seidenberg} that $R$ is a regular ring. Hence, $R_M$ is a regular local ring, and therefore a Cohen-Macaulay ring. Then $$G(M_M) = \mbox{ht}(M_M) = \mbox{ht}(M) = n.$$ Moreover, by \cite[Theorem 135]{K}, we obtain $$G(M) = G(M_M) = n > 1.$$
	Applying now \cite[Theorem 2]{Davis}, we obtain that $M$ has a minimal set $\{p_1, \ldots, p_m\}$ of generators such that $\{p_i\}$ generates a prime ideal for any $1 \leq i \leq m$ (since $G(M) > 1$). Set $x = p_k$ for some $1 \leq k \leq m$ and $P = Rx$. By Principal Ideal Theorem (see \cite[Theorem 10.2]{E}), we have $1 < \mbox{ht}(M) \leq m$, and consequently $P \subsetneq M$. Since $x$ is a prime element and $P = Rx$ is a non-maximal prime ideal of $R$, it follows from Theorem \ref{aneldsimples} that $R[\theta; \delta]$ does not satisfy $(\diamond)$.
	
	($\Leftarrow$) It follows from Proposition \ref{simplesdiamantedim1}.
\end{proof}

The above result has the following immediate consequence.

\begin{cor}\label{exemplosdsimplesdiamante}
	Let $R = K[x_1, \ldots , x_n]$ or $R = K[x_1^{\pm 1}, \ldots , x_n^{\pm 1}]$ with a $K$-derivation $\delta$ such that $R$ is $\delta$-simple. Then $R[\theta; \delta]$ satisfies $(\diamond)$ if and only if $n = 1$.
\end{cor}

Examples of UFD of Krull dimension 1 that are $\delta$-simple for some derivation $\delta$ are given below.

\begin{examples}\label{anellaurentdsimples}
	(1) Given $\alpha \in K\setminus \{0\}$, it is easy to see that $K[x]$ and $K[[x]]$ are $\delta$-simple with respect to the $K$-derivation $\delta = \alpha\partial_x.$

\vspace{.2cm}

	(2) Consider the local ring $K[x]_{M}$, where $M = K[x]x$. Since $K[x]$ is $\delta$-simple with the $K$-derivation $\delta = \alpha\partial_x$, where $\alpha \in K\setminus \{0\}$, we obtain that $K[x]_{M}$ is also $\delta$-simple.
\end{examples}

In the literature there are several examples of $\delta$-simple UFDs with Krull dimension bigger than $1$.

\begin{examples} (1) \cite[Example 13.4.5]{Nowicki} The polynomial ring $K[x, y]$ is $\delta$-simple with the $K$-derivation $$\delta = (x^2 + y^2 +2)\partial_{x} + (x^2 - y^2)\partial_{y}.$$

	\vspace{.2cm}
	
	(2) \cite[Example 13.4.1]{Nowicki} The polynomial ring $R = K[x_{1}, \ldots, x_{n}]$, where $n > 2$, is $\delta$-simple with the following $K$-derivation $$\delta = (1 - x_{1}x_{2})\partial_{x_{1}} + x_{1}^{3}\partial_{x_{2}} + \sum_{i = 3}^{n} x_{i-1}\partial_{x_{i}}.$$
		
	\vspace{.2cm}
	
	(3) \cite[Teorema 2.5.9]{Archer} The affine $\C$-algebra $R = \C[x_1, x_2, y_1, y_2]/\langle x_{1}^2 + y_{1}^2 - 1, x_{2}^2 + y_{2}^2 - 1\rangle$ with the $\C$-derivation $$\delta = ay_1\partial_{x_1} - ax_1\partial_{y_1} + by_2\partial_{x_2} - bx_2\partial_{y_2}$$ is $\delta$-simple if and only if $a/b$ is a non rational number. Note that $R \simeq \C[t_1, t_{1}^{-1}, t_2, t_{2}^{-1}]$ via the isomorphism
	$$\begin{array}{cc}
	x_1 + iy_1 \mapsto t_1,  \ \ \ \ & x_1 - iy_1 \mapsto t_1^{-1}\\
	x_2 + iy_2 \mapsto t_2,  \ \ \ \ & x_2 - iy_2 \mapsto t_2^{-1}.\\
	\end{array}$$
	
	\vspace{.2cm}
	
	(4) \cite[Proposition 3.3]{BLL} The local ring $R = K[x, y]_{\langle x, y \rangle}$ is $\delta$-simple with the $K$-derivation $$\delta = \partial_x + (\beta y^{n} + 1) \partial_y,$$ where $n$ is a positive integer and $\beta \in \Q \setminus \{0\}$. 
\end{examples}

The next examples give us an application of the above theorem for affine domains.

\begin{example}Let $R = K[x, y]/\langle x^2 + y^2 - 1 \rangle$ with derivation $\delta = y\partial_{x} - x\partial_{y}$. Then $R$ is $\delta$-simple (see \cite[Example 2.5]{Voskoglou}). Note that $R$ is an affine domain of Krull dimension 1. By Theorem \ref{dfudsimplesdiamante}, we concluded that $R[\theta; \delta]$ satisfies $(\diamond)$.
\end{example}

\begin{example}If $R = K[x, y, z]/\langle x^2 + yz - 1\rangle$ with derivation $\delta = (x^2y - z)\partial_{x} + 2x\partial_{y} - 2x^3\partial_{z}$, $R$ is $\delta$-simple (see \cite[Theorem 2.5.23]{Archer}). Since $R$ is an affine domain of Krull dimension $> 1$, by Theorem \ref{dfudsimplesdiamante}, $R[\theta; \delta]$ does not satisfy $(\diamond)$.
\end{example}

We finish this section with a result on locally nilpotent derivations.
Recall that a derivation $\delta$ of $R$ is called {\it locally nilpotent} if, for every $a \in R$, there exists $k > 0$ such that $\delta^k(a) = 0$. For example, let $R = K[x_{1}, \ldots, x_{n}]$. A class of locally nilpotent derivations of these rings are the so called {\it triangular} derivations, that is, derivations  satisfying $\delta(x_1) \in K$ and $\delta(x_i) \in K[x_{1}, \ldots, x_{i-1}]$ for all $i > 1$. Carvalho, Hatipo\u{g}lu and Lomp showed  that $R[\theta; \delta]$ satisfies $(\diamond)$ whenever $R$ is a commutative affine $K$-algebra with a locally nilpotent derivation (\cite[Proposition 2.1]{CHL}). We will show now that if $R$ is  a $\delta$-simple ring with respect to a locally nilpotent derivation, then we have $$\mbox{K.dim}(R) \leq \mbox{K.dim}(R[\theta; \delta]) \leq 1.$$

Given $K \subset L$ an extension of fields, denote by $\mbox{tr.deg}_K(L)$ the transcendence degree of $L$ over $K$ and $\mathcal{F}(R)$ the field of fractions of a ring $R$. If $R$ is an affine $K$-algebra which is an integral domain, it is known that the Krull dimension of $R$ coincides with the transcendence degree of $\mathcal{F}(R)$ over $K$, that is, $\mbox{K.dim}(R) = \mbox{tr.deg}_K(\mathcal{F}(R))$ (see \cite[Corollary 14.29]{Sharp}). In general, we have the following lemma whose proof is due to Manuel Reyes in MathOverflow \cite{R}. Since we were not aware of any reference for it in the literature we reproduce his proof here.


\begin{lem}\cite{R}\label{krulltranscendencia}
	Let $R$ be a $K$-algebra which is an integral domain. Then $$\mbox{K.dim}(R) \leq \mbox{tr.deg}_K(\mathcal{F}(R)).$$
\end{lem}
\begin{proof} Let $P_0 \subset P_1 \subset \ldots \subset P_m$ be a strictly ascending chain of prime ideals of $R$. We will show that $m \leq \mbox{tr.deg}_K(\mathcal{F}(R))$. For each $i \in \{1, \ldots, m\}$, we consider an element $x_i \in P_i \setminus P_{i-1}$. Let $R' \subseteq R$ be the $K$-subalgebra generated by $\{x_1, \ldots, x_m\}$ and set $P'_i = R' \cap P_i$ for all $0 \leq i \leq m$. Note that $P'_0 \subset P'_1 \subset \ldots \subset P'_m$ is a strictly ascending chain of prime ideals of $R'$ because $x_i \in P'_i \setminus P'_{i-1}$. Since $R'$ is a affine $K$-algebra, it follows from \cite[Corollary 14.29]{Sharp} that $\mbox{K.dim}(R') = \mbox{tr.deg}_K(\mathcal{F}(R'))$. Therefore $m \leq \mbox{K.dim}(R') = \mbox{tr.deg}_K(\mathcal{F}(R')) \leq \mbox{tr.deg}_K(\mathcal{F}(R))$, as desired.
\end{proof}

\begin{prop}\label{simpleslocalnilp}
	Let $R$ be a Noetherian integral domain without $\Z$-torsion and let $\delta$ be a derivation of $R$. If $R$ is $\delta$-simple and $\delta$ is locally nilpotent, then $$\mbox{K.dim}(R) \leq \mbox{K.dim}(R[\theta; \delta]) \leq 1.$$
\end{prop}
\begin{proof} First assume that $R$ is a field. Since $0$ is a maximal ideal of $R$ which is $\delta$-ideal, we can apply \cite[Theorem 2.10]{GW_Krull} to obtain $\mbox{K.dim}(R[\theta;\delta]) = 1$. Thus $0 = \mbox{K.dim}(R) < \mbox{K.dim}(R[\theta; \delta]) = 1$.
	
	On the other hand, suppose that $R$ is not a field and let $M$ be any maximal ideal of $R$. Since $R$ is $\delta$-simple, we must have $\delta(M) \nsubseteq M$. Furthermore, as $R$ has no $\Z$-torsion, we have also $\mbox{char}(R/M) = 0$. By \cite[Theorem 2.10]{GW_Krull}, $\mbox{K.dim}(R) = \mbox{K.dim}(R[\theta; \delta])$. Now, consider the subring of constants $R^{\delta}$ of $R$ which is a field because $R$ is $\delta$-simple. Set $K = R^{\delta}$ and it follows from Lemma \ref{krulltranscendencia} that $\mbox{K.dim}(R) \leq \mbox{tr.deg}_K(\mathcal{F}(R))$. Since $\delta$ is locally nilpotent, it follows from \cite[Lemma 4]{L} that $\mbox{tr.deg}_K(\mathcal{F}(R)) = 1$. Therefore $\mbox{K.dim}(R[\theta; \delta]) = \mbox{K.dim}(R) \leq \mbox{tr.deg}_K(\mathcal{F}(R)) = 1$.
\end{proof}


\section{Commutative Noetherian $\delta$-primitive rings}

In \cite{CHL}, Carvalho, Hatipo\u{g}lu and Lomp studied the property $(\diamond)$ in $R[\theta; \delta]$ whenever $R$ is a $\delta$-primitive ring. In particular, they showed the following.

\begin{thm}\label{CHLdprimitivo}\cite[Theorem 3.5]{CHL}
	Let $R$ be a Noetherian integral domain without $\Z$-torsion with non-zero derivation $\delta$ such that $R$ is $\delta$-primitive. If $R[\theta; \delta]$ satisfies $(\diamond)$, then $R$ is $\delta$-simple.
\end{thm}

In this section, we will show that if $R$ is an affine $K$-algebra that is an integral domain with non-zero derivation $\delta$ such that $R$ is $\delta$-primitive, then $R[\theta; \delta]$ satisfies ($\diamond$) if and only if $R$ is $\delta$-simple and \mbox{K.dim}$(R) \leq 1$. Moreover, we will characterize the differential operator rings $R[\theta; \delta]$ satisfying $(\diamond)$ when $R$ has Krull dimension 1.

A ring $R$ is called {\it $\delta$-$G$-ring} if $R$ is $\delta$-prime and $$\bigcap\{P \triangleleft R \,|\, P \ \mbox{is non-zero $\delta$-prime}\} \neq 0.$$ In \cite{GW}, Goodearl and Warfield showed that $R[\theta;\delta]$ is a primitive ring if and only if $\delta \neq 0$ and $R$ is either $\delta$-primitive or a $\delta$-$G$-ring (see \cite[Theorem 3.7]{GW}). In the case when $R$ is an affine $K$-algebra which is an integral domain, we have the following stronger result.

\begin{thm}\label{primitivoeqdprimitivo}
	Let $R$ be an affine $K$-algebra which is an integral domain with a derivation $\delta$. If $R$ is a $\delta$-$G$-anel, then $R$ is $\delta$-primitive. Consequently \begin{center}
		$R[\theta;\delta]$ is primitive if and only if $\delta \neq 0$ and $R$ is $\delta$-primitive.
	\end{center}
\end{thm}
\begin{proof}
	Assume that $R$ is a $\delta$-$G$-ring but is not $\delta$-primitive and take $M$ an arbitrary maximal ideal of $R$. Then there must exist a non-zero $\delta$-ideal $I$ such that $M \supseteq I$. Consider $(M: \delta)$ the largest $\delta$-ideal of $R$ contained in $M$. Note that $(M: \delta) \neq 0$ because $(M: \delta) \supseteq I$. Since $\mbox{char}(R/M) = 0$ (because $R$ is a $K$-algebra) and $M$ is a prime ideal, it follows that $(M: \delta)$ is a prime ideal of $R$. Thus
	
\begin{eqnarray*}
		0 \ \ \neq \ \ \bigcap\{P \triangleleft R \,|\, P \ \mbox{is non-zero $\delta$-prime}\} & \subseteq & \bigcap_{M \,\in\, \mbox{Max}(R)}\!\!\!\!\!\!\!\!\!(M: \delta)  \\
		& \subseteq & \bigcap_{M \,\in\, \mbox{Max}(R)}\!\!\!\!\!\!\!\!\!M \ \ \ \ \ \ \ \ \  = \ \  \mathcal{J}(R). \\
	\end{eqnarray*}
	But, as $R$ is an affine $K$-algebra which is an integral domain, by \cite[Theorem 5.3]{LAM}, we obtain $\mathcal{J}(R) = 0$, a contradiction. Therefore $R$ is $\delta$-primitive.
	
	The last statement is now a consequence of \cite[Theorem 3.7]{GW}.
\end{proof}

Taking into account the Propositions \ref{simplesdiamantedim1} and \ref{CHLdprimitivo} it follows easily the following.

\begin{cor}\label{dsimplesdiamante}
Let $R$ be a Noetherian integral domain of Krull dimension $\leq 1$ without $\Z$-torsion. If $R$ is $\delta$-primitive for some non-zero derivation $\delta$, then
	$$R[\theta; \delta] \ \mbox{satisfies} \ (\diamond) \ \ \mbox{if and only if} \ \  R \ \mbox{is} \ \mbox{$\delta$-simple}.$$
\end{cor}

\begin{cor}\label{diamanteprimitivo}
Let $R$ be an affine $K$-algebra which is an integral domain with a non-zero derivation $\delta$ such that $R[\theta; \delta]$ is primitive (or equivalently, $R$ is $\delta$-primitive). Then $$R[\theta; \delta] \ \mbox{satisfies} \ (\diamond) \ \ \mbox{if and only if} \ \  R \ \mbox{is} \ \mbox{$\delta$-simple and K.dim}(R) \leq 1 .$$
\end{cor}
\begin{proof}
($\Rightarrow$) Suppose that $R[\theta; \delta]$ satisfies $(\diamond)$. Since $R[\theta; \delta]$ is primitive, by Theorem \ref{primitivoeqdprimitivo}, it follows that $R$ is $\delta$-primitive and, by Theorem \ref{CHLdprimitivo}, $R$ is $\delta$-simple. The Theorem \ref{dfudsimplesdiamante} implies that $\mbox{K.dim}(R) \leq 1$.

($\Leftarrow$) It follows from Theorem \ref{dfudsimplesdiamante}.	
\end{proof}

The following is a straightforward consequence of Corollary \ref{diamanteprimitivo}.

\begin{cor}\label{exemplosdiamanteprimitivo}
Let $R = K[x_1, \ldots , x_n]$ or $K[x_1^{\pm 1}, \ldots , x_n^{\pm 1}]$ with non-zero $K$-derivation $\delta$ such that $R[\theta; \delta]$ is primitive. Then $$R[\theta; \delta] \ \mbox{satisfies} \ (\diamond) \ \ \mbox{if and only if} \ \ R \ \mbox{is} \ \mbox{$\delta$-simple and} \ n = 1 .$$
\end{cor}

\begin{cor}\label{quocienteprimitivo}
Let $K$ be an algebraically closed field, $R$ be an affine $K$-algebra which is an integral domain and $S = R[\theta; \delta]$. If $\delta$ is a locally nilpotent derivation of $R$, then every primitive quotients of $S$ have Krull dimension $\leq 1$.
\end{cor}
\begin{proof}
Suppose that $\delta$ is a locally nilpotent derivation of $R$. It follows from \cite[Proposition 2.1]{CHL} that $S$ satisfies $(\diamond)$. Hence, so does its primitive quotients. Let $\mathcal{P}$ be any primitive ideal of $S$. By \cite[Theorem 3.22]{GW_Livro}, either $\mathcal{P} = S(\mathcal{P} \cap R)$ or $S/\mathcal{P}$ is commutative. The latter case implies that $S/\mathcal{P}$ is a field, that is, $\mbox{K.dim}(S/\mathcal{P}) = 0$. If $\mathcal{P} = S(\mathcal{P} \cap R)$, then $S/\mathcal{P} = S/S(\mathcal{P} \cap R) = R/(\mathcal{P} \cap R)[\theta; \overline{\delta}]$, where $\overline{\delta}$ is the derivation of $R/(\mathcal{P} \cap R)$ induced by $\delta$. Since $R/(\mathcal{P} \cap R)[\theta; \overline{\delta}]$ is primitive and satisfies $(\diamond)$, it follows from Corollary \ref{diamanteprimitivo} that $R/(\mathcal{P} \cap R)$ is $\overline{\delta}$-simple and $\mbox{K.dim}(R/(\mathcal{P} \cap R)) \leq 1$. By \cite[Theorem 2.10]{GW_Krull}, we have  $\mbox{K.dim}(R/(\mathcal{P} \cap R)[\theta; \overline{\delta}]) = \mbox{K.dim}(R/(\mathcal{P} \cap R)) \leq 1$. 	
\end{proof}

The following result provides a sufficient condition for a ring $R$ to be $\delta$-primitive when $R$ has Krull dimension one.

\begin{lem}\label{dprimitivokdimum}
Let $R$ be a $K$-algebra which is an integral domain of Krull dimension $1$ and $\delta$ a derivation of $R$. If there exists a maximal ideal $M$ of $R$ such that $\delta(M) \nsubseteq M$, then $R$ is $\delta$-primitive.
\end{lem}
\begin{proof}
Suppose that there exists a maximal ideal $M$ of $R$ such that $\delta(M) \nsubseteq M$. We claim that $M$ does not contain non-zero $\delta$-ideals. Otherwise, let $I$ be a non-zero $\delta$-ideal contained in $M$ and consider the $\delta$-ideal $(M : \delta)$. Since $M$ is not a $\delta$-ideal, we must have $(M : \delta) \subsetneq M$. Thus $$0 \subsetneq I \subseteq (M : \delta) \subsetneq M.$$ But, as $M$ is prime and $\mbox{char}(R/M) = 0$, we obtain that $(M : \delta)$ is prime (\cite[Proposition 1.1]{GW}). Hence, $\mbox{ht}(M) > 1$, a contradiction. Therefore $M$ does not contain non-zero $\delta$-ideals and the result follows.
\end{proof}

The following lemma is known in the literature and its proof can be obtained in \cite[Theorem 2.3.8]{Archer} or in \cite[Corollary 2.12]{GW_Krull}. We will denote by $\mbox{Max}(R)$ the maximal spectrum of a ring $R$.

\begin{lem}\cite[Theorem 2.3.8]{Archer}\label{dideaismaximais}
	Let $R$ be a commutative affine $K$-algebra with $K$-derivation $\delta$ and $M \in \mbox{Max}(R)$. Then $\delta(M) \subseteq M$ if and only if $\delta(R) \subseteq M$.
\end{lem}

The next result characterize differential operator rings $R[\theta; \delta]$ satisfying $(\diamond)$ when $R$ is an affine $K$-algebra which is an integral domain of Krull dimension $1$. 

\begin{thm}\label{kdimumdiamante}
	Let $R$ be an affine $K$-algebra which is an integral domain of Krull dimension $1$ with non-zero $K$-derivation $\delta$. Then $R$ is $\delta$-primitive. Consequently,
	$$R[\theta; \delta] \ \mbox{satisfies} \ (\diamond) \ \ \mbox{if and only if} \ \  R \ \mbox{is} \ \mbox{$\delta$-simple}.$$
\end{thm}
\begin{proof}
	We claim that there exists a maximal ideal which is not a $\delta$-ideal. Otherwise, $\delta(M) \subseteq M$ for all $M \in \mbox{Max}(R)$. Since $\delta$ is a $K$-derivation, by Lemma \ref{dideaismaximais}, we would have $\delta(R) \subseteq M$ for all $M \in \mbox{Max}(R)$, and so $$\delta(R) \subseteq \bigcap\{M \triangleleft R \,|\, M \in \mbox{Max}(R) \} = {\mathcal J}(R).$$ But, by \cite[Theorem 5.3]{LAM}, ${\mathcal J}(R) = 0$, and therefore $\delta = 0$, a contradiction. Therefore there exists a maximal ideal which is not $\delta$-ideal and, by Lemma \ref{dprimitivokdimum}, the result follows.
	
	The result is now a consequence of Corollary \ref{diamanteprimitivo}.
\end{proof}

As an application we completely classify the differential operator rings $K[x,x^{-1}][\theta; \delta]$ which satisfy the property $(\diamond)$.

\begin{cor}\label{laurentdiamante}
For any non-zero $K$-derivation $\delta$ of $K[x,x^{-1}]$ such that $\delta_{\mid_{K[x]}}$ is a $K$-derivation of $K[x]$, the following statements are equivalent:
	\begin{enumerate}
		\item $K[x,x^{-1}][\theta; \delta]$ satisfies $(\diamond)$;
		\item $\delta(x) = \alpha x^n$ for some $\alpha \in K \setminus \{0\}$ and $n \geq 0$;
		\item $K[x,x^{-1}]$ is $\delta$-simple.
	\end{enumerate}
\end{cor}
\begin{proof}Let $R = K[x,x^{-1}] = K[x]A^{-1} = K[x]_x$, where $A = \{x^i \,|\, i\geq 0\}$, and recall that any $K$-derivation of $K[x]$ is of the form $\delta = a \partial_x$ for some $a \in K[x]$. Moreover, note that any derivation $\delta$ of $K[x]$ extends uniquely to a derivation of $R$ via the quotient's rule $\delta \left(rs^{-1}\right) = (\delta(r)s - r\delta(s))s^{-2}$ for all $rs^{-1} \in R$.

$(1) \Leftrightarrow (3)$ Follows from Theorem \ref{kdimumdiamante}.

$(2) \Rightarrow (3)$ Suppose that $\delta(x) = \alpha x^n$ for some $\alpha \in K \setminus \{0\}$ and $n \geq 0$. If $n = 0$, that is, $\delta(x) = \alpha$, where $\alpha \in K\setminus \{0\}$, then $R = K[x]_x$ is $\delta$-simple. Now consider the case when $n > 0$. Note that the non-zero $\delta$-prime ideals of $K[x]$ are precisely the maximal ideals of $K[x]$ that are $\delta$-ideals (\cite[Corollary 1.4]{GW}). Hence, by \cite[Proposition 2.8]{GW}, it is enough to show that every maximal ideal of $K[x]$ which is a $\delta$-ideal contains $x$. Given any maximal ideal $M$ which is a $\delta$-ideal, it follows that $M = K[x]p$, for some irreducible element $p \in K[x]$, and $p \mid \delta(p)$. But $p \mid \delta(p) = \alpha x^n \partial_x(p)$ implies $p \mid x$, and therefore $x \in K[x]p = M$.  
	
$(3) \Rightarrow (2)$ Suppose that $\delta(x) = a$, for some $a \in K[x]\setminus\{0\}$, such that $a \neq \alpha x^n$ for all $\alpha \in K \setminus \{0\}$ and $n \geq 0$. Consider the $\delta$-ideal $I = K[x]a$. Note that $I \cap A = \varnothing$, and as $a \neq 0$, we obtain $IA^{-1}$ is a non-trivial $\delta$-ideal of $R$. This shows that $R$ is not $\delta$-simple.
\end{proof}

\begin{rem}\label{diamantenaolocal}
Let $K$ be an algebraically closed field of characteristic zero and let $R$ be a commutative affine $K$-algebra with derivation $\delta$. In \cite[Proposition 2.1]{CHL}, Carvalho, Hatipo\u{g}lu and Lomp showed that $\delta$ being locally nilpotent implies that $R[\theta; \delta]$ satisfies $(\diamond)$. The converse is not true in general. Indeed, consider $R = K[x,x^{-1}]$ with the $K$-derivation $\delta = x\partial_x$. Corollary \ref{laurentdiamante} shows that $R[\theta; \delta]$ satisfies $(\diamond)$, but $\delta$ is not locally nilpotent, since $\delta^n(x) = x \neq 0$ for all $n \geq 0$.
\end{rem}


\section{Affine Algebras of Krull Dimension 2}

In the present section, we drop the condition of primitivity in $S = R[\theta; \delta]$. If $S$ satisfies $(\diamond)$, so does its primitive quotients. Given a commutative Noetherian $K$-algebra $R$ with a derivation $\delta$ and $\mathcal{P}$ a primitive ideal of $S$, either $\mathcal{P} = (\mathcal{P} \cap R)S$ or $S/\mathcal{P}$ is commutative (see \cite[Theorem 3.22]{GW_Livro}). If $\mathcal{P}$ is primitive such that $\mathcal{P} = (\mathcal{P} \cap R)S$, conditions for $S/\mathcal{P}$ to satisty $(\diamond)$ have been presented in the previous section. In some cases, the study of $(\diamond)$ can be reduced to the study of $(\diamond)$ on its primitives quotients.

\begin{lem}\label{quocientesprimitivoshl}\cite[Lemma 2.5]{HL}
	Suppose that $S$ is a Noetherian algebra such that every primitive ideal $\mathcal{P}$ of $S$ contains an ideal $\mathcal{Q}\subseteq \mathcal{P}$ which has a normalizing sequence of generators and $S/\mathcal{Q}$ satisfies $(\diamond)$. Then $S$ satisfies $(\diamond)$.
\end{lem}

In this section, we consider the case when $R$ is an affine $K$-algebra of Krull dimension 2. The main result is the Theorem \ref{primitivodiamante2} when $R = \C[x,y]$. We will denote by $\mathcal{F}_{\delta}$ the set of all $\delta$-ideals of $R$, that is, $\mathcal{F}_{\delta} = \{I\triangleleft R \,|\, I \ \mbox{is a $\delta$-ideal}\}$. We begin by considering the case when $R$ does not have maximal ideals which are $\delta$-ideals, that is, $\mbox{Max}(R) \cap \mathcal{F}_{\delta} = \varnothing$.

\begin{lem}\label{krull}
Let $R$ be a commutative Noetherian $K$-algebra of Krull dimension $2$ with derivation $\delta$ and $S = R[\theta; \delta]$. If $P$ is a non-zero prime $\delta$-ideal of $R$ which is not contained in any maximal ideal which is a $\delta$-ideal, then $\mbox{K.dim}_S(S/SP) = 1$.
\end{lem}
\begin{proof}
Let $P$ be as in the hypothesis. Note that $P$ is not maximal because $P$ is a $\delta$-ideal. By \cite[Proposition 2.7]{GW_Krull}, we have $\mbox{K.dim}_S(S/SP) = m + 1$, where $$m = \mbox{max}\{\mbox{K.dim}_S(S/SQ) \, | \, Q \in \mbox{Spec}(R) \ \mbox{and}\  Q \supsetneq P\}.$$ Since $R$ has Krull dimension $2$, the only prime ideals that properly contain $P$ are the maximal ones. Given any maximal ideal $M$ of $R$ containing $P$, by hypotheses, we must have $\delta(M) \nsubseteq M$. Moreover, as $\mbox{char}(R/M) = 0$, it follows  as in the proof of \cite[Lemma 2.4]{H} that $S/SM$ is a simple left $S$-module, and so $\mbox{K.dim}_S(S/SM) = 0$. This shows that $m = 0$. Therefore $\mbox{K.dim}_S(S/SP) = 1$, as desired.
\end{proof}

Recall that an element $a$ in a ring $S$ is called a {\it normal element} if $Sa = aS$. Note that if $S = R[\theta; \delta]$ and $a \in R$ is a Darboux element, then $a$ is a normal element in $S$. In fact, as $\delta(a) = ba$ for some $b \in R$, we have $\theta a = a\theta + \delta(a) = a(\theta + b)$. Thus $\theta^n a = a(\theta + b)^n$ and also $a\theta^n = (\theta - b)^na$, showing that $Sa = aS$.

The next result shows that the property $(\diamond)$ also can be related to the primitiveness of $S$ in our setting.

\begin{thm}\label{primitivodiamante}
Let $R$ be a UFD of Krull dimension $2$ which is an affine $K$-algebra. Suppose that $\delta$ is a derivation of $R$ such that $\mbox{Max}(R) \cap \mathcal{F}_{\delta} = \varnothing$. Then the following conditions are equivalent:
	\begin{enumerate}
		\item $S = R[\theta; \delta]$ satisfies $(\diamond)$;
		\item $R$ is not $\delta$-primitive;
		\item $S$ is not primitive.
	\end{enumerate}
\end{thm}
\begin{proof}
	$(1) \Rightarrow (2)$ Suppose that $S$ satisfies $(\diamond)$. Since $\mbox{K.dim}(R) = 2$, it follows from Corollary \ref{diamanteprimitivo} that $R$ is not $\delta$-primitive.
	
	$(2) \Leftrightarrow (3)$ Since $\mbox{Max}(R) \cap \mathcal{F}_{\delta} = \varnothing$, it follows that $\delta \neq 0$. The equivalence is now a consequence of Theorem \ref{primitivoeqdprimitivo}.
	
	$(3) \Rightarrow (1)$ Assume that $S$ is not primitive. Let $\mathcal{P}$ be any primitive ideal of $S$. Thus $\mathcal{P}$ is a non-zero prime ideal of $S$. It follows from \cite[Lemma 3.19 and Theorem 3.22 (a)]{GW_Livro} that $P = \mathcal{P}\cap R$ is a non-zero prime $\delta$-ideal of $R$ (note that $\delta \neq 0$). Moreover, by \cite[Theorem 3.22 (b)]{GW_Livro}, either $\mathcal{P} = SP$ or $\delta(R)\subseteq P$. The later case does not apply. In fact, suppose that $\delta(R)\subseteq P$ and consider a maximal ideal $M$ of $R$ such that $M \supseteq P$. Hence, $\delta(M) \subseteq \delta(R) \subseteq P \subseteq M$, contradicting the fact that $\mbox{Max}(R) \cap \mathcal{F}_{\delta} = \varnothing$. Thus $\mathcal{P} = SP$ and, by Lemma \ref{krull}, $S/\mathcal{P} = S/SP$ is a Noetherian domain of Krull dimension $1$, and therefore satisfies $(\diamond)$.
	
	Now, since $P \not\in \mbox{Max}(R)$, we must have $P = Rp$ for some irreducible Darboux elements $p \in R$. This shows that $\mathcal{P} = SP = S(Rp) = Sp$ is generated by a normal element. Applying Lemma \ref{quocientesprimitivoshl} we conclude that $S$ satisfies $(\diamond)$.
\end{proof}

\begin{prop}\label{contexemprimitivodiamante}
Let $R$ be an affine $K$-algebra which is an integral domain and let $\delta$ be a $K$-derivation of $R$. Suppose that there exist $M \in \mbox{Max}(R) \cap \mathcal{F}_{\delta}$ and a prime $\delta$-ideal $P$ such that $M \supseteq P$ and $\mbox{ht}(P) = \mbox{K.dim}(R) - 1$. If $S = R[\theta; \delta]$ satisfy $(\diamond)$, then $\delta(R) \subseteq P$.
\end{prop}
\begin{proof}
Suppose that $\delta(R) \nsubseteq P$. Then $\delta$ induces a non-zero $K$-derivation $\overline{\delta}$ in $R/P$. Since $P$ is a $\delta$-ideal and $\delta(R) \nsubseteq P$, $P$ is not maximal by Lemma \ref{dideaismaximais}. Thus $P \subsetneq M$ and $\delta(M) \subseteq M$ implies that $M/P$ is a non-trivial $\overline{\delta}$-ideal of $R/P$, and therefore $R/P$ is not $\overline{\delta}$-simple. Moreover, by \cite[Corollary 13.4]{E}, $\mbox{K.dim}(R/P) = \mbox{K.dim}(R) - \mbox{ht}(P)$ and, since $\mbox{ht}(P) = \mbox{K.dim}(R) - 1$, we obtain $\mbox{K.dim}(R/P) = 1$. Using Theorem \ref{kdimumdiamante}, we conclude that $S/SP \simeq R/P[\theta; \overline{\delta}]$ does not satisfy $(\diamond)$. Therefore $S$ does not satisfy $(\diamond)$.	
\end{proof}


Next, we present an example of a ring $R$ which is not $\delta$-primitive and such that $R[\theta; \delta]$ does not satisfy $(\diamond)$.

\begin{example}
Suppose that $K$ is an algebraically closed field and consider $R = K[x,y]$ with the $K$-derivation $\delta = x\partial_x + y\partial_y$. Since $K$ is an algebraically closed field, every maximal ideal of $R$ is of the form $M = R(x-\lambda) + R(y-\mu)$ for some $\lambda, \mu \in K$. If $\lambda = 0$, then $M$ contains the Darboux element $x$. In the case when $\lambda \neq 0$, we have $$\mu x - \lambda y = \mu x - \mu \lambda + \lambda \mu - \lambda y = \mu(x-\lambda) - \lambda(y-\mu) \in M$$
and $\mu x - \lambda y$ is a Darboux element. This shows that $R$ is not $\delta$-primitive. On the other hand, we have that $Rx + Ry \in \mbox{Max}(R) \cap \mathcal{F}_{\delta}$ and contains the irreducible Darboux element $x$. Moreover $\delta(R) \nsubseteq Rx$, since $y = \delta(y) \in \delta(R)$ and $y \not\in Rx$. By Proposition \ref{contexemprimitivodiamante}, we conclude that $R[\theta; \delta]$ does not satisfy $(\diamond)$.
\end{example}

The existence of Darboux elements plays, not only an importante role in the study of the property $(\diamond)$, but also in the study in $\delta$-$G$-rings as next propositions show.

\begin{prop}\label{dganel}
Let $R$ be a Noetherian UFD with derivation $\delta$. If there are infinitely many non-associated irreducible Darboux elements in $R$, then $R$ is not a $\delta$-$G$-ring.
\end{prop}
\begin{proof}
By \cite[Proposition 2.9]{GW}, it is enough to show that $R_c$ is not $\delta$-simple for all $c\in R \setminus \{0\}$. Given any non-zero element $c$ of $R$, then $c$ can be written as a finite product of irreducible elements. Since there are infinitely many irreducible Darboux elements, there must be an irreducible Darboux element $p$ that does not divide $c$. Let $\pi: R\rightarrow R_c$ be the canonical map $a\mapsto a1^{-1}$. Thus $\pi(p)$ is not invertible in the localization $R_c$, and so $R_c\pi(p)$ is a non-trivial $\delta$-ideal, that is,  $R_c$ is not $\delta$-simple, as desired.
\end{proof}

\begin{prop}\label{infdarbouxeqdganel}
Let $R$ be a UFD of Krull dimension $2$ which is an affine $K$-algebra and let $\delta$ be a non-zero $K$-derivation of $R$. Then there are only a finite number of non-associated irreducible Darboux elements in $R$ if and only if $R$ is a $\delta$-$G$-ring.
\end{prop}
\begin{proof}
($\Rightarrow$) First note that, since $R$ has Krull dimension $2$, given any non-zero $\delta$-prime ideal of $R$, it is either an ideal generated by an irreducible Darboux element or a maximal ideal which is a $\delta$-ideal. Thus considering $${\mathcal A} = \{Rp \,|\, p \ \mbox{is irreducible Darboux}\} \ \ \mbox{and} \ \ {\mathcal B} = \mbox{Max}(R) \cap \mathcal{F}_{\delta},$$ we have $\bigcap\{P \triangleleft R \,|\, P \ \mbox{is non-zero $\delta$-prime}\} = \bigcap_{P \in {\mathcal A} \cup {\mathcal B}} P$. Moreover, $\delta(b) \neq 0$ for some $b \in R$, since $\delta \neq 0$. It follows from Lemma \ref{dideaismaximais} that $0 \neq \delta(b) \in M$ for all $M\in {\mathcal B}$, that is, $0 \neq \delta(b) \in \bigcap_{M \in {\mathcal B}} M$.
	
If $R$ has no irreducible Darboux elements, that is, ${\mathcal A} = \varnothing$, then $$0 \neq \delta(b) \in \bigcap_{M \in {\mathcal B}} M = \bigcap\{P \triangleleft R \,|\, P \ \mbox{is non-zero $\delta$-prime}\}.$$ 

In the case when there are only a finite number $s > 0$ of non-associated irreducible Darboux elements in $R$, namely $p_1, \ldots, p_s$, we have $$0 \neq p_1\cdots p_s\delta(b) \in \bigcap_{P \in {\mathcal A} \cup {\mathcal B}} P = \bigcap\{P \triangleleft R \,|\, P \ \mbox{is non-zero $\delta$-prime}\}.$$ 

Hence, in both cases, we must have $\bigcap\{P \triangleleft R \,|\, P \ \mbox{is non-zero $\delta$-prime}\} \neq 0$. Furthermore, since $R$ is a domain, we have that $R$ is $\delta$-prime. Therefore $R$ is a $\delta$-$G$-ring.
	
($\Leftarrow$) Follows from Proposition \ref{dganel}.
\end{proof}

Combining Proposition \ref{infdarbouxeqdganel} with Theorem \ref{primitivoeqdprimitivo}, we obtain easily the following.

\begin{prop}\label{finitodarboux2}
Let $R$ be a UFD of Krull dimension $2$ which is an affine $K$-algebra and let $\delta$ be a non-zero $K$-derivation of $R$. If there are only a finite number of non-associated irreducible Darboux elements in $R$, then $R$ is $\delta$-primitive.
\end{prop}

Let $R = K[x, y]$ with a $K$-derivation $\delta$. Using Lemma \ref{dideaismaximais}, we can conclude that $$\mbox{Max}(R) \cap \mathcal{F}_{\delta} = \varnothing \ \ \ \mbox{if and only if} \ \ \ R = \langle \delta(x), \delta(y) \rangle.$$
A trivial example of $K$-derivations $\delta$ of $R = K[x,y]$ satisfying  $R = \langle \delta(x), \delta(y) \rangle$ are those such that $\delta(x) \in K\setminus\{0\}$ or $\delta(y) \in K\setminus\{0\}$.

A $K$-derivation $\delta$ of $K[x, y]$ is said to be a {\it Shamsuddin derivation} if $\delta$ is of the form $\delta = \partial_{x} + (ay + b)\partial_{y}$, where $a, b \in K[x]$. For these derivations, $R[\theta; \delta]$ does not satisfy $(\diamond)$.

\begin{example}\label{exemploanelpolnaodiamante}
	Let $R = K[x, y]$ with the $K$-derivation $\delta = \partial_{x} + (ay + b)\partial_{y}$, where $a, b \in K[x]$ and $a \neq 0$. Then $R[\theta; \delta]$ does not satisfy $(\diamond)$. In fact, if $R$ is $\delta$-simple, the result follows from Corollary \ref{exemplosdsimplesdiamante}. In the case when $R$ is not $\delta$-simple, it follows from \cite[Theorem 4.1, (b)]{BLL} that there exists a unique polynomial $c \in K[x]$ such that $\delta(c) = ac + b$ and $P = R(y - c)$ is the unique non-zero prime $\delta$-ideal of $R$. This shows that $y - c$ is the unique irreducible Darboux element of $R$. By Proposition \ref{finitodarboux2}, we conclude that $R$ is $\delta$-primitive. Hence, Theorem \ref{primitivodiamante} shows that $S$ does not satisfy $(\diamond)$.	
\end{example}

\begin{example}
Suppose that $K$ is algebraically closed and consider $R = K[x, y]$ with the $K$-derivation $$\delta = \partial_x + x^m(y + \gamma)^n\partial_y,$$ where $m, n \in \N$ and $\gamma \in K$. Then $S = R[\theta;\delta]$ satisfies $(\diamond)$ if and only if $n \neq 1$. Indeed, if $n = 1$, the result follows from Example \ref{exemploanelpolnaodiamante}. Conversely, if $n = 0$, then $\delta = \partial_x + x^m\partial_y$ is a triangular derivation. Thus $\delta$ is locally nilpotent, and so $S$ satisfies $(\diamond)$. Assume $n > 1$, and consider the elements
 $$p = y + \gamma \qquad \mbox{and} \qquad q_{\omega} = (x^{m+1} + \omega)(y + \gamma)^{n-1} + \frac{(m+1)}{(n-1)}, \ \ \mbox{where} \ \omega \in K.$$ We have $$\delta(p) = x^m (y + \gamma)^{n-1}p\qquad  \mbox{and} \qquad \delta(q_{\omega}) = (n-1) x^m(y + \gamma)^{n-1}q_{\omega},$$ and so they are Darboux elements.
Now, given any maximal ideal $M$ of $R$. Since $K$ is algebraically closed, we have $M = \langle x-\lambda, y-\mu \rangle$ for some $\lambda, \mu \in K$. If $\mu = -\gamma$, then $M = \langle x-\lambda, y+\gamma \rangle$ contains $p = y + \gamma$. In the case when $\mu \neq -\gamma$, consider the Darboux element  $$q_{\omega} = \left(x^{m+1} - \lambda^{m+1} - \frac{(m+1)}{(n-1) (\mu + \gamma)^{n-1}}\right)(y + \gamma)^{n-1} + \frac{(m+1)}{(n-1)}$$ for $\omega = - \lambda^{m+1} - \frac{(m+1)}{(n-1) (\mu + \gamma)^{n-1}}$. We have
\begin{eqnarray*}
	q_{\omega} &=& (y + \gamma)^{n-1}(x^{m+1}-\lambda^{m+1}) - \frac{(m+1)}{(n-1)(\mu + \gamma)^{n-1}}\left((y + \gamma)^{n-1}-(\mu + \gamma)^{n-1}\right)\\
	&=& a(x - \lambda) + b(y - \mu) \in M,
\end{eqnarray*}
where $$a = (y + \gamma)^{n-1}\left(\sum_{i=0}^{m} \lambda^i x^{m-i}\right) \ \mbox{and} \ b = - \frac{(m+1)}{(n-1) (\mu + \gamma)^{n-1}}\left(\sum_{i=0}^{n-2} (\mu + \gamma)^i (y + \gamma)^{n-2-i}\right).$$
Thus any maximal ideal of $R$ contains a Darboux element, and consequently a non-zero $\delta$-ideal. Therefore $R$ is not $\delta$-primitive. Moreover, $R = \langle \delta(x), \delta(y) \rangle$ because $\delta(x) = 1 \in K\setminus\{0\}$. Hence, $\mbox{Max}(R) \cap \mathcal{F}_{\delta} = \varnothing$ and the result follows from Theorem \ref{primitivodiamante}.
\end{example}

From now on, we assume that $K = \C$, the field of complex numbers, and $R = \C[x, y]$ with a $\C$-derivation $\delta$. Consider $L = \C(x, y)$ the field of fractions of $R$. Then $\delta$ extends uniquely to $L$. Let $$L^{\delta} = \{p/q \in L \,|\, \delta(p/q) = 0 \}$$ be the subring of constants of $L$. The main purpose here, is to provide necessary and sufficient conditions for the differential operator rings $\C[x,y][\theta; \delta]$ to satisfy $(\diamond)$.

\begin{prop}\label{caractdprimitivo}
Let $R = \C[x, y]$ with non-zero $\C$-derivation $\delta$. The following statements are equivalent:
	\begin{enumerate}
		\item $R[\theta; \delta]$ is primitive;
		\item $R$ is a $\delta$-$G$-ring;
		\item There are only a finite number of non-associated irreducible Darboux elements in $R$;
		\item $R$ is $\delta$-primitive;
		\item There is a maximal ideal that does not contain irreducible Darboux elements;
		\item $L^{\delta} = \C$.
	\end{enumerate}
\end{prop}
\begin{proof}
The equivalences (1) $\Leftrightarrow$ (4) and (2) $\Leftrightarrow$ (3) are a consequence of Theorem \ref{primitivoeqdprimitivo} and Proposition \ref{infdarbouxeqdganel}, respectively. The statement (3) $\Rightarrow$ (4) follows from Proposition \ref{finitodarboux2} and (4) $\Rightarrow$ (5) is obvious.
	
(5) $\Rightarrow$ (6) Suppose that $L^{\delta} \neq \C$. We will show that every maximal ideal of $R$ contains an irreducible Darboux element. Let $p, q \in R$ be such that $p/q \in L^{\delta} \setminus \C$. Without loss of generality, we can assume $\mbox{mdc}(p,q) = 1$. Since $\delta(p/q) = (\delta(p)q - p\delta(q))/q^2 = 0$, we have $\delta(p)q = p\delta(q)$, and so $p \mid \delta(p)q$. As $\mbox{mdc}(p, q) = 1$, we obtain $p \mid \delta(p)$. Thus there exists $c \in R$ such that $\delta(p) = cp$, and therefore $\delta(q) = cq$. In this case, the elements of the form $\alpha p - \beta q$ are Darboux elements for all $\alpha, \beta \in \C$, not both zero.

Let $M = \langle x - \lambda, y - \mu \rangle$, where $\lambda, \mu \in \C$, be any maximal ideal of $R$. Note that $$M = \{r(x, y) \in R \,|\, r(\lambda, \mu) = 0\}.$$
If $p(\lambda, \mu) = 0$, then $M$ contains the Darboux element $p$. In the case when $p(\lambda, \mu) \neq 0$, set $\alpha = q(\lambda, \mu)$ and $\beta =  p(\lambda, \mu) \neq 0$. Thus $r = \alpha p - \beta q \in R$ is a Darboux element such that $r \in M$.
	
The above argument shows that every maximal ideal $M$ contains a Darboux element. Since the irreducible factors of a Darboux element are also Darboux elements, $M$ must contains an irreducible Darboux element.
	
(6) $\Rightarrow$ (3) By \cite[Darboux's Theorem, p. 686]{M}, if there are infinity non-associated irreducible Darboux elements in $R$, then there exists $w \in L\setminus \C$ such that $\delta(w) = 0$, that is, $L^{\delta} \neq \C$.
\end{proof}

\begin{thm}\label{primitivodiamante2}
Let $R = \C[x,y]$ with non-zero $\C$-derivation $\delta$ and let $S = R[\theta; \delta]$. Then $S$ satisfies $(\diamond)$ if and only if $S$ is not primitive and for any $M \in \mbox{Max}(R) \cap \mathcal{F}_{\delta}$ and $p\in M$, where $p$ is an irreducible Darboux element, we have $\delta(R) \subseteq Rp$.
\end{thm}
\begin{proof}
Suppose that $S$ satisfaz $(\diamond)$. It follows from Corollary \ref{exemplosdiamanteprimitivo} that $S$ is not primitive. Assume, in addition, that $\mbox{Max}(R) \cap \mathcal{F}_{\delta} \neq \varnothing$. By contradiction, suppose that there exist $M \in \mbox{Max}(R) \cap \mathcal{F}_{\delta}$ and $p\in M$, where $p$ is an irreducible Darboux element, such that $\delta(R) \nsubseteq Rp$. Applying Proposition \ref{contexemprimitivodiamante}, we would have that $S$ does not satisfy $(\diamond)$, a contradiction.
	
Conversely, if $S$ is not primitive and $\mbox{Max}(R) \cap \mathcal{F}_{\delta} = \varnothing$, it follows from Theorem \ref{caractdprimitivo} that $S$ satisfies $(\diamond)$. Suppose now that $S$ is not primitive and for any $M \in \mbox{Max}(R) \cap \mathcal{F}_{\delta}$ and every irreducible Darboux element $p$ contained in $M$, we have $\delta(R) \subseteq Rp$. In this case, we will show that every primitive ideal $\mathcal{P}$ of $S$ contains an ideal $\mathcal{Q}$ of $S$ such that $\mathcal{Q}$ is generated by normal elements and $S/\mathcal{Q}$ satisfies $(\diamond)$. In fact, let $\mathcal{P}$ be any primitive ideal of $S$. Since $S$ is not primitive, $\mathcal{P}$ is a non-zero prime ideal of $S$. Thus $P = \mathcal{P}\cap R$ is a non-zero prime $\delta$-ideal of $R$ and either $\mathcal{P}=SP$ or $\delta(R)\subseteq P$. We will study these two cases below.
		
\underline{Case $\delta(R)\subseteq P$:} If $P$ is not maximal, then $P = Rp$ for some irreducible Darboux element $p \in R$. In this case, set $\mathcal{Q} = SP = Sp \subseteq \mathcal{P}$ and note that $\mathcal{Q}$ is an ideal of $S$ generated by the normal element $p$.  Since $\delta(R)\subseteq P$, we obtain that $$S/\mathcal{Q} = S/SP \simeq R/P[\theta; \overline{\delta}] = R/P[\theta]$$ is a commutative Noetherian ring, and therefore satisfies $(\diamond)$. On the other hand, when $P$ is a maximal ideal, as $S$ is not primitive, it follows from Proposition \ref{caractdprimitivo} that there exists a irreducible Darboux element $p$ such that $p \in P$. Set $\mathcal{Q} = Sp \subseteq \mathcal{P}$ and note that again $\mathcal{Q}$ is an ideal of $S$ generated by the normal element $p$. Since $P \in \mbox{Max}(R) \cap \mathcal{F}_{\delta}$ and $p$ is an irreducible Darboux element contained in $P$, by hypothesis, we have $\delta(R) \subseteq Rp$. Thus $$S/\mathcal{Q} = S/Sp \simeq R/Rp[\theta; \overline{\delta}] = R/Rp[\theta]$$ is again a commutative Noetherian ring, and so it satisfies $(\diamond)$.
		
\underline{Case $\delta(R)\nsubseteq P$:} In this case, we have $\mathcal{P} = SP$. Moreover, $P$ is not maximal. Otherwise, $P$ would be a maximal ideal which is a $\delta$-ideal and, by Lemma \ref{dideaismaximais}, we would have $\delta(R)\subseteq P$, which does not occur. Hence, $P = Rp$ for some irreducible Darboux element $p \in R$, and so $\mathcal{P} = SP = Sp$ is generated by the normal element $p$. Furthermore, any maximal ideal $M$ of $R$ containing $P = Rp$ is not a $\delta$-ideal because, otherwise by hypothesis, we would have $\delta(R) \subseteq Rp = P$. It follows from Lemma \ref{krull} that $\mbox{K.dim} (S/\mathcal{P}) = \mbox{K.dim}(S/SP) = 1$, and therefore $S/\mathcal{P}$ satisfies $(\diamond)$. In this case, it is enough to set $\mathcal{Q} = \mathcal{P}$.

By Lemma \ref{quocientesprimitivoshl}, we conclude that $S$ satisfies $(\diamond)$.		 	
\end{proof}

\begin{cor}\label{primitivodiamante3}
Let $R = \C[x,y]$ with $\C$-derivation $\delta$ such that $\mbox{mdc}(\delta(x), \delta(y)) = 1$. Then $R[\theta; \delta]$ satisfies $(\diamond)$ if and only if $S$ is not primitive and $\mbox{Max}(R) \cap \mathcal{F}_{\delta} = \varnothing$.
\end{cor}
\begin{proof}
Since $\mbox{mdc}(\delta(x), \delta(y)) = 1$ and $\delta(x), \delta(y) \in \delta(R)$, we have $\delta(R) \nsubseteq Rp$, for any irreducible Darboux element $p \in R$. By Theorem \ref{primitivodiamante2}, the result follows.
\end{proof}

Recall that a polynomial is called {\it homogeneous of degree} $n$ if all its monomials have the same degree $n$. A derivation $\delta$ of $\C[x, y]$ is said to be {\it homogeneous derivation of degree} $n$ if $\delta(x)$ and $\delta(y)$ are both homogeneous of the same degree $n$.

\begin{example}
	Let $R = \C[x, y]$ with $\C$-derivation $\delta$. Suppose that $\delta$ is homogeneous of degree $n$ and $\mbox{mdc}(\delta(x), \delta(y)) = 1$. Then $R[\theta; \delta]$ satisfies $(\diamond)$ \ if and only if \ $n = 0$.
In fact, assume that $n > 0$ and set $M = \langle x, y \rangle$. Since $\delta(x)$ and $\delta(y)$ are homogeneous polynomials of the same degree $n > 0$, we have that $\delta(x)$ and $\delta(y)$ does not have constant terms. This implies that $\delta(x), \delta(y) \in M$, and consequently $M$ is a maximal ideal of $R$ which is a $\delta$-ideal. Thus $\mbox{Max}(R) \cap \mathcal{F}_{\delta} \neq \varnothing$. By Corollary \ref{primitivodiamante3}, $S$ does not satisfy $(\diamond)$. Conversely, it is enough to observe that if $n = 0$, $\delta$ is locally nilpotent.
\end{example}

\begin{rem}
	Let $R = \C[x,y]$ with $\C$-derivation $\delta$ such that $\mbox{mdc}(\delta(x), \delta(y)) = c \in R \setminus \C$, and so $\delta = c (a\partial_{x} + b\partial_{y}) = c \delta'$, where $a, b \in R$, $\mbox{mdc}(a, b) = 1$ e $\delta' = a\partial_{x} + b\partial_{y}$. In this case, we have:
	
	\begin{enumerate}
		\item $\mbox{Max}(R) \cap \mathcal{F}_{\delta} \neq \varnothing$.
		
		In fact, if  $\mbox{Max}(R) \cap \mathcal{F}_{\delta} = \varnothing$, then $R = \langle \delta(x), \delta(y) \rangle$, and therefore $\mbox{mdc}(\delta(x), \delta(y)) = 1$, a contradiction.
		
		\item If $R[\theta; \delta]$ satisfies $(\diamond)$, then $c \in \bigcap \{M \in \mbox{Max}(R) \,|\, \delta(M) \subseteq M \}$.
		
		Indeed, suppose that there is a maximal ideal $M$ of $R$ which is a $\delta$-ideal and $c \not\in M$. Since $R[\theta; \delta]$ satisfies $(\diamond)$, $R$ is not $\delta$-primitive and, by Proposition \ref{caractdprimitivo}, there exists an irreducible Darboux element $p \in R$ such that $p \in M$. As $c \not\in M$, we have $p \nmid c$. Moreover, $\mbox{mdc}(a, b) = 1$ implies that either $p \nmid a$ or $p \nmid b$. Thus either $p \nmid ca = \delta(x)$ or $p \nmid cb = \delta(y)$. This shows that either $\delta(x) \not\in Rp$ or $\delta(y) \not\in Rp$, and therefore $\delta(R) \nsubseteq Rp$. Applying now Theorem \ref{primitivodiamante2}, we obtain that $R[\theta; \delta]$ does not satisfy $(\diamond)$, a contradiction.
		
		\item Using (2) above, we can construct examples of differential operator rings $R[\theta; \delta]$ which do not satisfy $(\diamond)$. For instance, choose the polynomials $a$ and $b$ belonging to a maximal ideal $M$ such that $c \not\in M$. For example, set $c = x$. We know that $c \not\in M = \langle x - 1, y \rangle$. Now, set $a = y \in M$ and $b = x -1 \in M$. Thus $M$ is a maximal ideal of $R$ which is a $\delta$-ideal because $\delta (x - 1) = \delta (x) = ca \in M$ and $\delta (y) = cb \in M$. As $c \not\in M$, the remark above shows that $\C[x, y][\theta; x(y\partial_{x} + (x-1)\partial_{y}]$ does not satisfy $(\diamond)$.
	\end{enumerate}
\end{rem}

In the next example, we also have that $\mbox{Max}(R) \cap \mathcal{F}_{\delta} \neq \varnothing$. In this case, we will show that $S$ is not primitive and for any $M \in \mbox{Max}(R) \cap \mathcal{F}_{\delta}$ and $p\in M$, where $p$ is an irreducible Darboux element, we have $\delta(R) \subseteq Rp$, and apply Theorem \ref{primitivodiamante2}.

\begin{example}
Let $R = \C[x,y]$ with the $\C$-derivation $\delta = [(x+1)y - 1]y(\partial_{x} - y^2\partial_{y})$.
	
We start by describing the irrebucible Darboux elements of $R$. Set $\delta' = \partial_{x} - y^2\partial_{y}$, and so $\delta = [(x+1)y - 1]y\delta'$. Note that if $p$ is an irreducible Darboux element of $\delta$, from $p \mid \delta(p) = [(x+1)y - 1]y\delta'(p)$, it follows that $p \mid (x+1)y - 1$, $p \mid y$ or $p \mid \delta'(p)$. In the latter case, $p$ is an irreducible Darboux element of $\delta'$. Let $\delta'_{\mid_{\C[x]}}=\partial_x$ be the ordinary derivation of polynomials in $\C[x]$ and use the classical notation for derivatives, that is, $\delta'(a)=a'$ for all $a\in \C[x]$.
	
Suppose that $p=\sum_{i=0}^n a_i y^i$ is a non-constant Darboux element of $\delta'$ with $a_i \in \C[x]$ and $a_n\neq 0$. Thus $n > 0$. Otherwise $p \in \C[x]$ would be a Darboux  elements and as the only Darboux elements of $\C[x]$ are the non-zero constants, $p$ would be constant. By computing $\delta'(p)$, we obtain
$$\delta'(p)= \sum_{i=0}^n a_i' y^i - \sum_{i=1}^n ia_i y^{i+1} = a_0' + a_1'y + \sum_{i=2}^n (a_i' - (i-1)a_{i-1})y^i - na_ny^{n+1}.$$
Since $p$ is a Darboux element of $\delta'$, there exists $c \in R$ such that $\delta'(p) = cp$. Moreover, as $p$ is a polynomial of degree $n$ in $y$ that divides $\delta'(p)$, we must have that $c$ is a polynomial of degree $1$ in $y$, and so $c =  b_0 + b_1y$ with $b_0, b_1\in \C[x]$ and $b_1 \neq 0$. Hence,
$$cp = \sum_{i=0}^n b_0a_i y^i + \sum_{i=0}^n b_1a_i y^{i+1} = b_0a_0 + (b_0a_1 + b_1a_0)y + \sum_{i=2}^n (b_0a_i + b_1a_{i-1})y^i + b_1a_ny^{n+1}.$$
Comparing the coefficients of the monomials in $y$ in both expressions, we have
	\begin{eqnarray}
		a_0' &=& b_0a_0 \\
		a_1' &=& b_0a_1+b_1a_0\\
		a_i' - (i-1)a_{i-1}&=& b_0a_i + b_1a_{i-1} \qquad \forall\, 2\leq i\leq n\\
		-na_n &=& b_1a_n.
	\end{eqnarray}
From the last equation it follows that $b_1=-n \in \C\setminus \{0\}$. We will show that $b_0$ is zero. Suppose $b_0\neq 0$. By (1), $a_0$ is a Darboux element in $\C[x]$, and so $a_0\in \C$. But then $0=\delta(a_0) = a_0'=b_0a_0$ and $a_0=0$. Let  $1\leq i\leq n$ be the least index with  $a_i\neq 0$. From (3), it follows that $a_i' = b_0a_i$, that is, $a_i$ is a Darboux element in $\C[x]$, and again $a_i\in \C$. Thus $0 = a_i' = b_0a_i$, and consequently $b_0 = 0$ or $a_i = 0$, a contradiction. Therefore $b_0=0$, $c=-ny$ and $(-ny)p = d(p)$. From equations (1) to (3), we have
$$a_0' = 0, \qquad a_1' = -n a_0, \qquad  a_i'  = (i-1-n)a_{i-1},  \qquad \forall \, 2\leq i\leq n.$$
Consider $a = a_n$. Then, for any $1\leq k\leq n$,
	 $$a_{n-k} = \frac{-1}{k} a_{n-(k-1)}'  = \frac{-1}{k} \frac{(-1)^{k-1}}{(k-1)!} a^{(k)} =  \frac{(-1)^k}{k!} a^{(k)}.$$
Hence, the Darboux elements of $\delta'$ have the form $p = \sum_{k=0}^n  \frac{(-1)^k}{k!} a^{(k)} y^{n-k}$ for some polynomial $a\in \C[x]$ . Since $a^{(n+1)} = (-1)^n n! a_0'=0$, we have $\mbox{deg}(a) \leq n$. On the other hand, it is easy to check that for any element $p$ of the above form, we have $\delta'(p) = (-ny)p$, and so also is a Darboux element of $\delta'$. Therefore, a non-constant element $p \in R$ is a Darboux element of $\delta'$ if and only if $$p = \sum_{k=0}^n  \frac{(-1)^k}{k!} a^{(k)} y^{n-k}$$ for some non-zero polynomial $a\in \C[x]$ with $\mbox{deg}(a)\leq n$. Factoring out the leading coefficient of $a$, one can assume that $a$ is monic.

Let $p= \sum_{k=0}^n  \frac{(-1)^k}{k!} a^{(k)} y^{n-k}$ be a non-constant Darboux element of $\delta'$ of degree $n$ in $y$. Suppose $a \in \C[x]$ is a monic polynomial such that $2 \leq \mbox{deg}(a) \leq n$. Since $\C$ is algebraically closed, $a = uv$ for some $u, v\in \C[x,y]$ such that $u$ is monic of degree 1. Now, taking into account that $u' = 1$, for any $k > 0$,
$$a^{(k)} = (uv)^{(k)} = \sum_{i=0}^k {k \choose i} u^{(i)} v^{(k-i)} = uv^{(k)} + k v^{(k-1)}.$$
Using the fact that $v$ is a polynomial of degree at most $n-1$, $v^{(n)}=0$ and
	\begin{eqnarray*}p &=& uvy^n + \sum_{k=1}^n  \frac{(-1)^k}{k!} \left(uv^{(k)} + k v^{(k-1)}\right) y^{n-k}\\
		&=&uy \left(\sum_{k=0}^{n-1}  \frac{(-1)^k}{k!}v^{(k)}y^{n-k-1} \right) +  \sum_{k=1}^n  \frac{(-1)^k}{(k-1)!} v^{(k-1)} y^{n-k}\\ &=& (uy-1)\left(\sum_{k=1}^n  \frac{(-1)^{k-1}}{(k-1)!} v^{(k-1)} y^{n-k}\right).
	\end{eqnarray*}
Hence, if $p$ is an irreducible Darboux element of $\delta'$, $\mbox{deg}(a) < 2$. If $a$ is constant, then $a=1$ and $p = y^n$. Being irreducible forces $p = y$. If $a$ has degree one, then $a'=1$ and $p =ay^n - y^{n-1}$. Since  $p$ is irreducible, $p = (x+\mu)y - 1$ for some $\mu \in \C$. This shows that every irreducible Darboux elements of $\delta'$ (and therefore of $\delta$) is a scalar multiple of one of the following: $$p = y \qquad \mbox{or} \qquad q_{\mu} = (x+\mu)y - 1 \qquad \mbox{for some }\mu \in \C.$$

Note that $q_{\mu_{1}} = (x + \mu_1)y -1$ and $q_{\mu_{2}} = (x + \mu_2)y -1$ are non-associated whenever $\mu_{1} \neq \mu_{2}$. Therefore there are infinitely many non-associated irreducible Darboux elements of $\delta$. By Proposition \ref{caractdprimitivo}, we have that $S$ is not primitive.

We show now that if there exists an irreducible Darboux element $p$ in $M \in \mbox{Max}(R) \cap \mathcal{F}_{\delta}$, then $\delta(R) \subseteq Rp$. Let $M = \langle x-\alpha , y-\beta \rangle$ be a maximal ideal which is a $\delta$-ideal for some $\alpha,\beta\in \C$. Thus $\delta(x-\alpha) = [(x+1)y - 1]y \in M$, and so $(x+1)y - 1 \in M$ or $y \in M$. If $(x+1)y - 1 \in M$, $(\alpha+1)\beta - 1 = 0$ and $M = \left\langle x-\alpha, y - \frac{1}{\alpha + 1}\right\rangle$. If $y \in M$, $\beta = 0$ and $M = \langle x-\alpha, y\rangle$. Therefore, the maximal ideals which are $\delta$-ideals are of the form
$$M = \langle x-\alpha, y \rangle \qquad \mbox{or} \qquad M = \left\langle x-\alpha, y - \frac{1}{\alpha + 1}\right\rangle.$$ If $M = \langle x-\alpha, y \rangle$, as $q_{\mu}(x, y) = (x + \mu)y - 1 \not\in M$, the only irreducible Darboux elements contained in $M$ are the scalar multiples of $p(x, y) = y$. Note that $\delta(R) \subseteq Ry$, since $\delta(x), \delta(y) \in Ry$. If $M = \left\langle x-\alpha, y - \frac{1}{\alpha + 1}\right\rangle$, obviously $p(x, y) = y \not\in M$. Moreover, $$q_{\mu}(x, y) = (x + \mu)y - 1 \in M \ \  \Leftrightarrow \ \  q_{\mu}\left(\alpha, \frac{1}{\alpha + 1}\right) = (\alpha + \mu)\frac{1}{\alpha + 1} - 1 = 0 \ \ \Leftrightarrow \ \ \mu = 1.$$
Hence, the only irreducible Darboux elements contained in $M$ are the scalar multiples of $q_{1}(x, y) = (x + 1)y - 1$. Note that $\delta(R) \subseteq R[(x + 1)y - 1]$ because $\delta(x), \delta(y) \in R[(x + 1)y - 1]$.

Therefore, it follows from Theorem \ref{primitivodiamante2} that $R[\theta; \delta]$ satisfies $(\diamond)$.
\end{example}

\section*{Acknowledgments}

This paper is part of the doctoral thesis of the second author at the  Federal University of Rio Grande do Sul, Brazil. The second author wishes to thank his co-supervisor Paula A. A. B. Carvalho for her guidance, assistance and understanding. Both authors would also like to thank Christian Lomp for all of his comments and suggestions. Finally, they would like to thank Rene Baltazar for fruitful discussions and Marcelo Escudeiro Hernandes for suggesting several references needed for this work.

\end{document}